\documentclass[11pt]{article}
\usepackage{graphicx}
\usepackage{amsmath, amsthm, amssymb, tikz, bm}
\usepackage{mathtools}
\usepackage{xifthen}
\usepackage{comment, array}

\usepackage[margin=1.1in]{geometry} 

\usepackage[shortlabels]{enumitem}
\usepackage{todonotes, accents}
\usepackage[colorlinks=true,
linkcolor=blue,citecolor=blue,
urlcolor=blue]{hyperref}

\usetikzlibrary{calc,shapes, backgrounds}
\usetikzlibrary{patterns}
\usetikzlibrary{decorations.markings}

\newtheorem{theorem}{Theorem}[section]
\newtheorem{lemma}[theorem]{Lemma}

\newtheorem{claim}{Claim}

\newtheorem{conjecture}{Conjecture}

\newcommand{\lp}{\left (}
\newcommand{\rp}{\right )}

\newcommand{\hh}{\hat{H}}

\newcommand{\diam}{\text{diam}}
\newcommand{\ind}{\text{ind}}

\newcommand{\ordiam}{\overrightarrow{\text{diam}}}
\newcommand{\ori}{\overrightarrow}

\DeclarePairedDelimiter{\abs}{\lvert}{\rvert}%

\makeatletter
\renewenvironment{description}%
               {\list{}{\leftmargin=0pt 
                        \labelwidth\z@ \itemindent-\leftmargin
                        }}%
               {\endlist}
\makeatother

\title{On the oriented diameter of graphs with given minimum degree}

\author{
Garner Cochran \thanks{Berry College, Mt. Berry, GA, 30149 
({\tt gcochran@berry.edu}).}
\and
Zhiyu Wang \thanks{Louisiana State University, Baton Rouge, LA, 70803
({\tt zhiyuw@lsu.edu}). This author was supported in part by LA Board of Regents grant LEQSF(2024-27)-RD-A-16.}
}
\begin{document}

\maketitle

\begin{abstract}
Erd\H{o}s, Pach, Pollack, and Tuza [\textit{J. Combin. Theory Ser. B, 47(1) (1989), 73--79}] proved that the diameter of a connected $n$-vertex graph with minimum degree $\delta$ is at most $\frac{3n}{\delta+1}+O(1)$.
The oriented diameter of an undirected graph $G$, denoted by $\overrightarrow{\text{diam}}(G)$, is the minimum diameter of a strongly connected orientation of $G$. Bau and Dankelmann [\textit{European J. Combin., 49 (2015), 126–133}] showed that for every bridgeless $n$-vertex graph $G$ with minimum degree $\delta$, $\overrightarrow{\text{diam}}(G) \leq \frac{11n}{\delta+1}+9$. They also showed an infinite family of graphs with oriented diameter at least $\frac{3n}{\delta+1} + O(1)$ and posed the problem of determining the smallest possible value $c$ for which $\overrightarrow{\text{diam}}(G) \leq c \cdot\frac{3n}{\delta+1}+O(1)$ holds. 
In this paper, we show that the smallest value $c$ such that the upper bound above holds for all $\delta\geq 2$ is $1$, which is best possible.
\end{abstract}

\section{Introduction}\label{sec:intro}
A directed graph $D = (V(D),E(D))$ is a graph with a vertex set $V(D)$ and an edge set $E(D)$ consisting of ordered pairs of vertices, called \textit{arcs} or directed edges. We use $uv$ to denote the arc $(u,v)$, i.e., the arc oriented from $u$ to $v$. Given an undirected graph $G=(V(G),E(G))$, an \textit{orientation} of $G$ is a directed graph such that each edge in $E(G)$ is assigned a direction. Given a (directed) graph $D$ and two vertices $u,v \in V(D)$, the \textit{distance} from $u$ to $v$ in $D$, denoted by $d_D(u,v)$, is the number of edges of a shortest (directed) path from $u$ to $v$ in $D$. More generally, given any two disjoint vertex subsets $S,T\subseteq V(D)$, define $d_D(S,T):= \min\{d_{D}(s,t): s\in S, t\in T\}$. If $S = \{s\}$, we simply write $d_D(S,T)$ as $d_D(s,T)$. 
For convenience, given a vertex $v$ and a subgraph $H$ of $G$, we also use $d_D(v,H)$ to denote $d_D(v,V(H))$.
We often ignore the subscript if there is no ambiguity on the underlying graph. Given a (directed) graph $D$, the \textit{diameter} of $D$ is defined to be $\diam(D)=\max\{d_D(u,v): u,v\in V(D)\}$. Given an undirected graph $G$ and $v\in V(G)$, let $N_G(v)$ denote the neighborhood of $v$ in $G$, and let $N_G[v]:= N_G(v)\cup \{v\}$. An edge $e \in E(G)$ is called a \textit{bridge} if $G-e$ is disconnected. A graph is called \textit{bridgeless} if it contains no bridge. 

A directed graph $D$ is called \textit{strongly connected} if for any two vertices $u,v \in V(D)$, there exists a directed path from $u$ to $v$. Robbins~\cite{Robbins1939}, showed in 1939 that every bridgeless graph has a strongly connected orientation. The \textit{oriented diameter} of a bridgeless graph $G$, denoted by $\ordiam(G)$, is defined as
$$\ordiam(G)=\min\{\diam(D):\textrm{$D$ is a strongly connected orientation of $G$}\}.$$ Note that for any bridgeless graph $G$, $\ordiam(G)\ge \diam(G)$.

Chv\'atal and Thomassen~\cite{Chvatal-Thomassen1978}, showed in 1978 that determining the oriented diameter of a given graph is NP-complete. In the same paper, Chv\'atal and Thomassen showed that every bridgeless graph $G$ with diameter $d$ satisfies $\ordiam(G)\leq 2d^2+2d$, and there exist bridgeless graphs of diameter $d$ for which every strong orientation has diameter at least $\frac{1}{2}d^2+d$. The upper bound was recently improved by Babu, Benson, Rajendraprasad and Vaka \cite{BBRV2021} to  $1.373d^2+6.971d-1$.

The paper by Chv\'atal and Thomassen \cite{Chvatal-Thomassen1978} has led to further investigations of such bounds on the oriented diameter with respect to other graph parameters, including the diameter \cite{FMR2004,Huang-Ye2007,KLW2010}, the radius \cite{CGT1985}, the domination number \cite{FMPR2004, Laetsch-Kurz2012}, the maximum degree \cite{DGS2018, Li-Chen2025}, the minimum degree \cite{Bau-Dankelmann2015,CDS2019,Surmacs2017}, the number of edges of the graph \cite{CCDS2021}, and other graph classes \cite{Chen-Chang2021,Gutin1994,GKTY2002,Gutin-Yeo2002,Huang-Ye2007, Koh-Ng2005,KRS2022, Lakshmi2011,Lakshmi-Paulraja2007, Lakshmi-Paulraja2009, Plesnik1985, Soltes1986, WCDGSV2021, GLW2023+}. For more results, see the survey by Koh and Tay \cite{Koh-Tay2022}.

In this paper, we are concerned about bounding the oriented diameter of a graph $G$ with minimum degree $\delta$. 
Erd\H{o}s, Pach, Pollack and Tuza \cite{EPPT1989} proved that the diameter of a connected graph of order $n$ and minimum degree $\delta$ is at most $\frac{3n}{\delta+1}+O(1)$. In 2015, Bau and Dankelmann \cite{Bau-Dankelmann2015} showed the following similar bound for the oriented diameter of a graph $G$ with minimum degree $\delta\geq 2$.

\begin{theorem}\label{thm:Bau-Dankelmann}\cite{Bau-Dankelmann2015}
Given a bridgeless graph $G$ of order $n$ and minimum degree $\delta \geq 2$, $\ordiam(G)\leq \frac{11n}{\delta+1}+9$. Moreover, given $\delta\geq 4$ and sufficiently large $n$, there exists a bridgeless $n$-vertex graph $G$ with minimum degree $\delta$ such that $\ordiam(G)\geq \frac{3n}{\delta+1}+O(1)$.
\end{theorem}

In the same paper, Bau and Dankelmann \cite{Bau-Dankelmann2015} posed the problem of determining the smallest possible value $c$ for which $\ordiam(G) \leq c \cdot  \frac{3n}{\delta+1}+O(1)$ holds. Theorem \ref{thm:Bau-Dankelmann} implies that $1\leq c\leq 11/3$. The upper bound on $c$  was improved to $7/3$ by Surmacs \cite{Surmacs2017} and was further improved recently by the first author \cite{Cochran2023+} to $2$. In this paper, we asymptotically answer Bau and Dankelmann's question by showing that the smallest value $c$ for which $\ordiam(G) \leq c \cdot\frac{3n}{\delta+1}+O(1)$ holds for all $\delta \geq 2$ is indeed $1$.
 
 \begin{theorem}\label{thm:closed-bound}
 For any $\epsilon > 0$, there exists some constant $C=C(\epsilon)$ such that for every bridgeless $n$-vertex graph $G$ with minimum degree $\delta \geq 3$ and $n$ sufficiently large, $\ordiam(G) \leq (3+\epsilon) \frac{n}{\delta-2}+ C$.
 \end{theorem}

 We remark that the proof of Theorem \ref{thm:closed-bound} can be easily converted to a polynomial-time algorithm (with respect to $n$), which for
a given bridgeless graph of order $n$ and minimum degree $\delta$, finds an orientation with the above diameter upper bound. 

We also believe that with more technical analysis, the denominator $\delta-2$ in Theorem \ref{thm:closed-bound} can be improved to $\delta-1$ with the same method. However, for clarity of the paper, we choose to prove the 
$(3+\epsilon) \frac{n}{\delta-2}+ C$ upper bound instead, as it has the same asymptotic implication when $\delta \to\infty$. In light of Theorem \ref{thm:closed-bound} and the result of Erd\H{o}s, Pach, Pollack and Tuza \cite{EPPT1989}, we also make the following conjecture.

\begin{conjecture}\label{conj:Cochran-Wang}
There exists a constant $C$ such that for every bridgeless graph $G$ of order $n$ and minimum degree $\delta\geq 2$,
$$\ordiam(G) \leq \frac{3n}{\delta+1}+C.$$
\end{conjecture}

We remark again that if true, the upper bound in Conjecture \ref{conj:Cochran-Wang} is tight (up to the constant), as shown by a construction given by Bau and Dankelmann \cite{Bau-Dankelmann2015}. For the self-completeness of the paper, we present another similar example of graphs of given minimum degree $\delta$ (where $\delta \geq 4$) and large order $n$, in which every strongly connected orientation has diameter at least $\frac{3n}{\delta+1}+O(1)$. Given vertex-disjoint graphs $G_1, G_2, \cdots, G_r$, define the \textit{sequential join} $G_1 + G_2 + \cdots + G_r$ of $G_1,G_2, \cdots, G_r$, as the graph obtained from their union by adding an edge from every vertex of $G_i$ to every vertex of $G_{i+1}$ for all $i\in [r-1]$. Given $\delta \geq 4$ and $k\in \mathbb{N}$, consider the graph 
\begin{align*}
    G_{\delta,k} &= K_{\delta-1} + K_2 + K_2 + K_{\delta-3} + K_2 + K_2 + K_{\delta-3} + \cdots  \\
                  & \quad \quad + K_2 + K_2 + K_{\delta-3} + K_2 + K_2 + K_{\delta-1},
\end{align*}
where $K_2 + K_2 + K_{\delta-3}$ is repeated $k$ times. It is easy to check that this graph has minimum degree $\delta$, order $n= (\delta-1) + (\delta+1) \cdot k + (\delta+3) = (\delta+1)(k+2)$, and diameter $3k+3$. Hence the $\ordiam(G_{\delta,k}) \geq \diam(G_{\delta,k}) = \frac{3n}{\delta+1}-3$. 

\section{Main Lemma}
 In this section, we show the main lemma needed for the proof of Theorem \ref{thm:closed-bound}. Given positive integers $s$ and $t$, let $[t]:= \{i \in \mathbb{Z}: 1\leq i\leq t\}$ and $[s,t]:= \{i\in \mathbb{Z}: s\leq i\leq t\}$. Given a $u-v$ path $P$, we use $\accentset{\circ}{P}$ to denote the \textit{interior} of the path $P$, i.e., $\accentset{\circ}{P}$ = $P-\{u,v\}$. 
 
 \begin{lemma}\label{lem:main-lemma}
Let $\epsilon>0$. Given a bridgeless graph $G$ of order $n$ and minimum degree $\delta \geq 3$, there exists some constant $L(\epsilon) >0$, a sequence of bridgeless oriented subgraphs $H_0 \subseteq H_1\subseteq H_2 \subseteq \cdots \subseteq H_m \subseteq G$ (for some $m \in \mathbb{N}$) and a sequence of vertex subsets $S_0 \subseteq S_1\subseteq S_2 \subseteq \cdots \subseteq S_m$ such that for each $i\in [m]$, $S_i\subseteq V(H_i)$ and 
 \begin{enumerate}[(a)]
     \item $\diam(H_i) \leq (3+\epsilon)|S_i|$;
     \item $\abs*{\bigcup_{v\in S_i} N[v]} \geq (\delta-2)|S_i|$;
     \item For all $v\in V(G)$, $d_G(v, H_m) \leq L(\epsilon)$.
 \end{enumerate}
 \end{lemma}

\begin{proof}[Proof of Lemma \ref{lem:main-lemma}]
Given $\epsilon> 0$, let $L(\epsilon) = 100/\epsilon.$
Fix an arbitrary vertex $s_0 \in V(G)$. 
Initially set $H_0 = G[\{s_0\}]$ and $S_0 = \{s_0\}$. Suppose we have defined $H_i$ and $S_i$ satisfying $(a)$ and $(b)$ of the lemma for some $i\geq 0$. If $d_G(v, H_i) \leq L(\epsilon)$ for all $v\in V(G)\backslash V(H_i)$, then we set $m = i$ and we are done. Otherwise, we will define $H_{i+1}$ and $S_{i+1}$ in the rest of the section and show that $(a)$ and $(b)$ hold for $H_{i+1}$ and $S_{i+1}$. The proof of $(c)$ follows easily from the terminating condition above. 

Let $P = u_0 u_1 u_2 \cdots u_p$ be a longest path in $G$ such that $p \equiv 0$ (mod $3$), $u_0\in V(H_i)$ and $d(u_j, u_0) = j$ for all $j\in [p]$.
Orient all edges $u_j u_{j+1}$ in $P$ from $u_j$ to $u_{j+1}$ for $0\leq j< p$. Let $Q = w_0 w_1\cdots w_q w_{q+1}$ be a shortest path from $u_p$ to some vertex in $H_i$ that is \textit{consistent} with the orientation of $\ori{P}$, i.e., none of the edges in $Q$ is a directed edge $\ori{u_{j+1}u_j}$ for any $j\in [0, p-1]$. Orient all edges $w_j w_{j+1}$ in $Q$ from $w_j$ to $w_{j+1}$ (if not yet oriented in $\ori{P}$). 
Note that $w_0 = u_p$ and $w_{q+1} \in V(H_i)$. For convenience, given a vertex $v \in V(P)$, define $\ind_P(v) = i$ if $v = u_i$. Similarly, given a vertex $v\in V(Q)$, define  $\ind_Q(v) = i$ if $v = w_i$.

Initially, let $A'_{i+1} = \{u_j \in V(P): j\in [p] \textrm{ and $j\equiv 0$ (mod $3$)}\}$ and let $B'_{i+1} = \{w_j \in V(Q): j\in [0, q-2] \textrm{ and $j\equiv 0$ (mod $3$)}\}$. Observe that $|E(P)|=3|A_{i+1}'|$ and $|E(Q)|\leq 3|B_{i+1}'|+2$.
For a technical reason, among all choices of $Q$, we pick $Q$ such that $|V(Q) \cap A_{i+1}'|$ is maximum. 
Given two distinct vertices $u$ and $w$, we say $u$ and $w$ are \textit{almost non-overlapping} if there are no two \textit{independent} (i.e., internally vertex-disjoint) paths of length at most two between $u$ and $w$, and \textit{overlapping} otherwise. Let $$A_{i+1}'' = \{v \in B'_{i+1}\backslash A'_{i+1}: \textrm{ $v$ and $a$ are almost non-overlapping for all } a\in A'_{i+1}\},$$ and
$$B_{i+1}'' = \{v \in A'_{i+1}\backslash B'_{i+1}: \textrm{ $v$ and $b$ are almost non-overlapping for all } b\in B'_{i+1}\}.$$

\begin{figure}[htb]
	\begin{center}
        	\resizebox{11cm}{!}{\input{highway.tikz}}
    \end{center} 
    \caption{The main paths $P$ and $Q$.}
    \label{fig:main_paths}
\end{figure}

Let $A_{i+1} = A_{i+1}' \cup A_{i+1}''$, $B_{i+1} = B_{i+1}' \cup  B_{i+1}''$. Let $S_{i+1}'$ be the larger set of $A_{i+1}$ and $B_{i+1}$ (arbitrary if same size) and $S_{i+1} = S_i \cup S_{i+1}'$. See Figure \ref{fig:main_paths} for an illustration. In Figure \ref{fig:main_paths}, the vertices in $A_{i+1}'$ are circled, and the vertices in $B_{i+1}'$ are boxed.

In order to travel between vertices in $P$ and $Q$ efficiently, we want to make sure that we could `transit' from vertices in $P$ to vertices in $Q$ and vice versa through certain short cuts. In particular, we want to show that vertices in $A_{i+1}'\backslash B_{i+1}''$, which are on $P$, can transit with vertices on $B_{i+1}'$ in a short distance; and similarly, vertices in $B_{i+1}'\backslash A_{i+1}''$, which are on $Q$, can transit with vertices on $A_{i+1}'$ in a short distance. 
The next three claims establish such properties. 

\begin{claim}\label{cl:short-cut}
For every $a \in A_{i+1}'\backslash B_{i+1}$, there exists some $b \in B_{i+1}'\backslash A_{i+1}$ such that 
$b$ is overlapping with $a$, i.e.,
(1) either $ab\in E(G)$ and there exists another path $a c_1 b$ of length $2$; or (2) $ab\notin E(G)$, and there exist two distinct paths $a c_1 b$ and $a c_2 b$ of length $2$ between $a$ and $b$.
Similarly, for every $b \in B_{i+1}'\backslash A_{i+1}$, there exists some $a\in A_{i+1}' \backslash B_{i+1}$ such that $a$ is overlapping with $b$.
\end{claim}
\begin{proof}
Let $a\in A_{i+1}'\backslash B_{i+1}\subseteq A_{i+1}'\backslash B_{i+1}'$. Suppose for contradiction that for every $b\in B_{i+1}'$, none of (1) and (2) happens. That implies $a$ is almost non-overlapping with all vertices in $B_{i+1}'$. It follows by the definition of $B_{i+1}''$ that $a \in B_{i+1}''$, contradicting that $a \notin B_{i+1}$. Hence there exists some $b \in B_{i+1}'$ such that either (1) or (2) holds.
Now we claim that $b \notin A_{i+1}$. Suppose otherwise that $b\in A_{i+1}= A_{i+1}'\cup A_{i+1}''$. Clearly $b\notin A_{i+1}'$ since $d(a,b) \leq 2$. Hence $b\in A_{i+1}''$. However, by definition of $A_{i+1}''$, $b$ is almost non-overlapping with all vertices in $A_{i+1}'$, contradicting that (1) or (2) holds for $a$ and $b$. Thus $b \in B_{i+1}'\backslash A_{i+1}$.

Let $b\in B_{i+1}'\backslash A_{i+1}$. Similar to before, there exists some $a\in A_{i+1}'$ such that either (1) or (2) holds. We claim that $a\notin B_{i+1}$. Suppose otherwise $a\in B_{i+1}= B_{i+1}'\cup B_{i+1}''$. Similar to before, $a \notin B_{i+1}''$, otherwise $a$ is almost non-overlapping with $b$, giving a contradiction. Hence $a\in B_{i+1}'$. Thus $a \in A_{i+1}'\cap B_{i+1}'$, which implies that $a \in V(P)\cap V(Q)$. Note that $b\in V(Q)$. By our definition of $B_{i+1}'$, $d_Q(a,b) \geq 3$. Thus since $a$ and $b$ are overlapping, it must follow that $N[a]\cap N[b]\subseteq V(P)$. Moreover, in the two independent paths $P_1, P_2$ of length at most $2$ between $a$ and $b$, there must be two directed edges $e_1$, $e_2 \in E(\ori{P})$ with $e_1\in E(P_1)$ and $e_2\in E(P_2)$
such that $e_1$ and $e_2$ are both oriented towards $b$ (if $\ind_Q(b)<\ind_Q(a)$) or away from $b$ (if $\ind_Q(a)<\ind_Q(b)$), as otherwise we will have $d_Q(a,b)\leq 2$ by our choice of $Q$. However, such orientation is impossible due to the minimality of $P$ (i.e., $P$ is a shortest path between $u_0$ and $u_p$). Hence $a \in A_{i+1}'\backslash B_{i+1}$. 
\end{proof}

\begin{claim}\label{cl:orientable}
\begin{enumerate}[(i)]
    \item For any pair of distinct vertices $a_1,a_2 \in A_{i+1}'\backslash B_{i+1}$ and any vertex $b_1 \in B_{i+1}'\backslash A_{i+1}$, every path of length at most $2$ between $a_1$ and $b_1$ is edge-disjoint with every path of length at most $2$ between $a_2$ and $b_1$.

    \item For any pair of distinct vertices $b_1,b_2\in B_{i+1}'\backslash A_{i+1}$ and any vertex $a_1 \in A_{i+1}'\backslash B_{i+1}$, if there exists an edge $e$ contained in both a path of length at most $2$ between $a_1$ and $b_1$ and a path of length at most $2$ between $a_1$ and $b_2$, then $e\in E(\ori{P})$.
\end{enumerate}
\end{claim}

\begin{figure}[htb]
	\begin{center}
        \begin{minipage}{.2\textwidth}
        		\resizebox{3cm}{!}{\begin{tikzpicture}[scale=1, Wvertex/.style={circle, draw=black, fill=white, scale=2}, bvertex/.style={circle, draw=black, fill=black, scale=0.4},rvertex/.style={circle, draw=red, fill=red, scale=0.2}]

\node [bvertex, label={[font=\small] right:$u$}] (u1) at (0,0) {};

\node [bvertex, label={[font=\small] left:$a_1$}] (a1) at (-1,-1.7) {};

\node [bvertex, label={[font=\small] right:$a_2$}] (a2) at (1,-1.7) {};

\node [bvertex] (c) at (-2,0) {};

\node [bvertex] (d) at (2,0) {};

\node [bvertex, label={[font=\small] above:$b_1$}] (b1) at (0,1.7) {};

\draw (b1) -- (u1) -- (a1)--(c)--(b1);
\draw (u1) --(a2) --(d) --(b1);

\end{tikzpicture}	}
        \end{minipage}
            \hspace{2cm}
        \begin{minipage}{.3\textwidth}
        		\resizebox{3cm}{!}{\begin{tikzpicture}[scale=1, Wvertex/.style={circle, draw=black, fill=white, scale=2}, bvertex/.style={circle, draw=black, fill=black, scale=0.4},rvertex/.style={circle, draw=red, fill=red, scale=0.2}]

\node [bvertex, label={[font=\small] right:$u\in V(P)$}] (u1) at (0,0) {};

\node [bvertex, label={[font=\small] left:$b_1$}] (b1) at (-1,-1.7) {};

\node [bvertex, label={[font=\small] right:$b_2$}] (b2) at (1,-1.7) {};

\node [bvertex] (c) at (-2,0) {};

\node [bvertex] (d) at (2,0) {};

\node [bvertex, label={[font=\small] above:$a_1\in V(P)$}] (a1) at (0,1.7) {};

\draw (a1) -- (u1) -- (b1)--(c)--(a1);
\draw (u1) --(b2) --(d) --(a1);

\end{tikzpicture}	}
        \end{minipage}
    \end{center} 
    \caption{Independent paths between $a_1, a_2$ and $b_1$, and between $b_1, b_2$ and $a_1$.}
    \label{fig:claim2}
\end{figure}

\begin{proof}
First, observe that, in (i), if there exists an edge contained in both a path of length at most $2$ between $a_1$ and $b_1$ and a path of length at most $2$ between $a_2$ and $b_1$, then $N[a_1]\cap N[a_2]\cap N[b_1]\neq \emptyset$. The same holds for $b_1, b_2$ and $a_1$ in (ii). See Figure \ref{fig:claim2} for one of the cases.

The proof of (i) is fairly straightforward. By our assumption, $a_1,a_2 \in V(A_{i+1}')$, which implies that $d_G(a_1,a_2)\geq 3$. Hence, $N[a_1]\cap N[a_2]\cap N[b_1]\subseteq N[a_1]\cap N[a_2]=\emptyset$. Thus (i) follows from the observation above.

Now we will show (ii). Similarly, by our assumption, $b_1,b_2 \in V(B_{i+1}')$. If $N[b_1]\cap N[b_2]= \emptyset$, then we are done, as $N[b_1]\cap N[b_2]\cap N[a_1] \subseteq N[b_1]\cap N[b_2]=\emptyset$ and (ii) follows from the observation above.
Hence we can assume $N[b_1]\cap N[b_2]\neq \emptyset$. Now observe that $N[b_1]\cap N[b_2]\subseteq V(P)$ since otherwise we will have $d_{G- E(P)}(b_1,b_2)\leq 2$, contradicting the minimality of $Q$. Hence, any edge $e = a_1 u$ that satisfies the assumption in (ii) must have $u\in N[b_1]\cap N[b_2]\cap N[a_1] \subseteq V(P)$. Since $a_1\in V(P)$, it follows from the minimality of $P$ that $e\in E(\ori{P})$.
\end{proof}

By Claim \ref{cl:short-cut}, for each $a\in A_{i+1}'\backslash B_{i+1}$, there exists some $b_a\in B_{i+1}'\backslash A_{i+1}$ such that there are two independent paths of length at most $2$ between $a$ and $b_a$; for each $b\in B_{i+1}'\backslash A_{i+1}$, there exists some $a_b \in A_{i+1}'\backslash B_{i+1}$ such that there are two independent paths of length at most $2$ between $b$ and $a_b$. For each $a\in A_{i+1}'\backslash B_{i+1}$, arbitrarily pick some $b_a$ satisfying the property above; similarly, for each $b\in B_{i+1}'\backslash A_{i+1}$, arbitrarily pick some $a_b$ satisfying the property above (if the pair $\{b, a_b\}$ has not been picked before).
For each pair $\{a,b\}$ above, we obtain two independent paths of length at most $2$ between $a$ and $b$.
Let $C_{i+1}$ be the collection of the vertices in all such paths that are not in $V(P)\cup V(Q)$.
Also let $R_{i+1}$ be the collection of the edges in all such paths. 
Define $$H_{i+1}'= \lp V(P) \cup V(Q)\cup C_{i+1}, E(\ori{P})\cup E(\ori{Q})\cup R_{i+1}\rp,$$ and $H_{i+1} = H_{i}\cup H_{i+1}'$. Note that some edges in $R_{i+1}$ may not be oriented yet. We will orient them so that every vertex in $A_{i+1}'\backslash B_{i+1}$ can transit with vertices in $B_{i+1}'$ by using few edges in $H_{i+1}'$, and similarly every vertex in $B_{i+1}'\backslash A_{i+1}$ can transit with vertices in $A_{i+1}'$ by using few edges in $H_{i+1}'$. However, note that such paths satisfying the property above may travel within $H_i$.
Instead of making a separate argument for those cases, for clarity of the proof, when analyzing $H_{i+1}= H_i\cup H_{i+1}'$, we will contract $H_{i}$ into a single vertex $s_i$ and call the resulting multigraph $\hat{H}_{i+1}$.
Note that $s_i$ is considered a vertex in both $P$ and $Q$. Thus, $\ori{s_i u_1}\in E(\ori{P})$ and $\ori{w_q s_i}\in E(\ori{Q})$ in $\hh_{i+1}$.
Observe that for any vertices $x,y\in V(H_{i+1}')$, there exists a path $P_{xy}$ from $x$ to $y$ in $H_{i+1}$ such that
$$|E(P_{xy}) \cap E(H_{i+1}')| = d_{\hh_{i+1}}(x,y).$$  

\begin{claim}\label{cl:good-orientaiton}
There exists a strongly connected orientation of $\hh_{i+1}$ such that 
for each $a\in A_{i+1}'\backslash B_{i+1}''$, 
\newcounter{enumcounter}
\begin{enumerate}[(i)]
    \item there exist $b_1\in B_{i+1}'\cup \{s_i\}$ with $d_{\hh_{i+1}}(a,b_1) \leq 5$, and
    \item $b_2 \in B_{i+1}' \cup \{s_i\}$ with $d_{\hh_{i+1}}(b_2,a)\leq 5$;
    
    \setcounter{enumcounter}{\value{enumi}}
\end{enumerate}
moreover, for each $b\in B_{i+1}'\backslash A_{i+1}''$, 
\begin{enumerate}[(i)]
\setcounter{enumi}{\value{enumcounter}}
\item there exist $a_1\in A_{i+1}'\cup \{s_i\}$ with $d_{\hh_{i+1}}(b,a_1)\leq 5$ and
\item $a_2\in A_{i+1}'\cup \{s_i\}$ with $d_{\hh_{i+1}}(a_2, b)\leq 5$.
\end{enumerate}
\end{claim}

\begin{proof}
For each $a\in A_{i+1}'\backslash B_{i+1}''$, if $a\in A_{i+1}'\cap B_{i+1}'$, then (i) and (ii) hold by letting $b_1= b_2 = a$. Similarly, for each $b\in B_{i+1}'\backslash A_{i+1}''$, if $b\in A_{i+1}'\cap B_{i+1}'$, then (iii) and (iv) hold by letting $a_1= a_2 = b$. Thus it suffices to prove Claim \ref{cl:good-orientaiton} by showing that there exists a strongly connected orientation of $\hh_{i+1}$ such that  (i) and (ii) hold for all $a\in A_{i+1}'\backslash B_{i+1}$, and (iii) and (iv) hold for all $b\in B_{i+1}'\backslash A_{i+1}$. 
In the rest of the proof, we will orient the undirected edges in $\hh_{i+1}$ and show (i) and (ii). The proofs for (iii) and (iv) are symmetric and follow similar arguments.

By Claim \ref{cl:short-cut}, for every $a\in A_{i+1}'\backslash B_{i+1}$, there exists some $b_a\in B_{i+1}'\backslash A_{i+1}$ with two independent paths $P_1, P_2$ between $a$ and $b_a$ of length at most $2$. We will orient the edges in $P_1$ and $P_2$ for each $a\in A_{i+1}'\backslash B_{i+1}$ and its corresponding $b_a$, and similarly for each $b \in B_{i+1}'\backslash A_{i+1}$ and its corresponding $a_b$. Note that some edges in $P_1$ and $P_2$ may have already been oriented due to the orientation in the paths $\ori{P}$ and $\ori{Q}$. We will orient the rest of the edges in $P_1$ and $P_2$ based on the already oriented edges. Observe that by Claim \ref{cl:orientable}, such paths $P_1$ and $P_2$ for different pairs of $a$ and $b_a$ (similarly $b$ and $a_b$) are either edge-disjoint or only share an oriented edge in $P$. Thus the new orientations assigned to one pair of paths will not affect the orientation assignment for another pair of paths. Let $C$ be the cycle formed by $P_1\cup P_2$. We assume that $P_1$ and $P_2$ are both of length $2$. The case when one of $P_1$ and $P_2$ is an edge follows similar reasoning and is simpler. Let $P_1 = ac_1 b_a$ and $P_2 = ac_2 b_a$.

{\bf Case $1$}: All edges in $C$ are not oriented or all oriented in the same direction (i.e., can be completed to a directed cycle). Then we can orient the undirected edges in $C$ so that all edges in $C$ are oriented in the same direction. Since $|C|\leq 4$, (i) and (ii) hold by letting both $b_1$ and $b_2$ be $b_a$. 

{\bf Case $2$}: Suppose the two edges incident to $a$ in $P_1$ and $P_2$ are both oriented towards $a$. That implies that one of these edges (let's say $\ori{c_1 a}$) is in $P$ and the other (i.e., $\ori{c_2 a}$) is in $Q$, which in turn implies that $a$ is in $V(P)\cap V(Q)$. Hence we could then find some $b_1\in B_{i+1}'\cup \{s_i\}$ such that $d(a,b_1)\leq 4$ by traveling along $\ori{Q}$. Observe that $\ori{c_2 b_a}$ is not a directed edge, since otherwise $c_2 b_a \in V(\ori{P})$, contradicting the minimality of $P$. Similarly, $\ori{c_1 b_a}$ is not a directed edge. It follows $c_1 b_a$ and $c_2 b_a$ are either undirected, or oriented towards $c_1$ and $c_2$ respectively. In either case, we can assign orientations $\ori{b_a c_1}$ and $\ori{b_a c_2}$ (if not oriented before) to the edges $b_a c_1$ and $b_a c_2$. We can then let $b_2$ in (ii) be $b_a$ since $d_{\hh_{i+1}}(b_a, a)\leq 2$. 

\begin{figure}[htb]
	\begin{center}
        \begin{minipage}{.15\textwidth}
        		\resizebox{2.5cm}{!}{\begin{tikzpicture}[scale=1, Wvertex/.style={circle, draw=black, fill=white, scale=2}, bvertex/.style={circle, draw=black, fill=black, scale=0.4},rvertex/.style={circle, draw=red, fill=red, scale=0.2}, decoration={markings, mark= at position 0.8 with {\arrow[scale=2]{latex}}}]

\node [bvertex, label={[font=\normalsize] right:$a$}] (a) at (0,1.7) {};

\node [bvertex, label={[font=\normalsize] left:$b_a$}] (ba) at (0,-1.7) {};

\node [bvertex, label={[font=\normalsize] left:$c_1$}] (c1) at (-1,0) {};

\node [bvertex, label={[font=\normalsize] right:$c_2$}] (c2) at (1,0) {};

\draw (a) -- (c1) -- (ba)--(c2)--(a);
\draw [postaction={decorate}] (c1) -- (a);
\draw [postaction={decorate}] (c2) -- (a);

\end{tikzpicture}	}
        \end{minipage}
            \hspace{0.2cm}
        \begin{minipage}{.15\textwidth}
        		\resizebox{2.5cm}{!}{\begin{tikzpicture}[scale=1, Wvertex/.style={circle, draw=black, fill=white, scale=2}, bvertex/.style={circle, draw=black, fill=black, scale=0.4},rvertex/.style={circle, draw=red, fill=red, scale=0.2}, decoration={markings, mark= at position 0.8 with {\arrow[scale=2]{latex}}}]

\node [bvertex, label={[font=\normalsize] right:$a$}] (a) at (0,1.7) {};

\node [bvertex, label={[font=\normalsize] left:$b_a$}] (ba) at (0,-1.7) {};

\node [bvertex, label={[font=\normalsize] left:$c_1$}] (c1) at (-1,0) {};

\node [bvertex, label={[font=\normalsize] right:$c_2$}] (c2) at (1,0) {};

\draw (a) -- (c1) -- (ba)--(c2)--(a);
\draw [postaction={decorate}] (c1) -- (a);
\draw [postaction={decorate}] (ba) -- (c2);

\end{tikzpicture}	}
        \end{minipage}
             \hspace{0.2cm}
        \begin{minipage}{.15\textwidth}
        		\resizebox{2.5cm}{!}{\begin{tikzpicture}[scale=1, Wvertex/.style={circle, draw=black, fill=white, scale=2}, bvertex/.style={circle, draw=black, fill=black, scale=0.4},rvertex/.style={circle, draw=red, fill=red, scale=0.2}, decoration={markings, mark= at position 0.8 with {\arrow[scale=2]{latex}}}]

\node [bvertex, label={[font=\normalsize] right:$a$}] (a) at (0,1.7) {};

\node [bvertex, label={[font=\normalsize] left:$b_a$}] (ba) at (0,-1.7) {};

\node [bvertex, label={[font=\normalsize] left:$c_1$}] (c1) at (-1,0) {};

\node [bvertex, label={[font=\normalsize] right:$c_2$}] (c2) at (1,0) {};

\draw (a) -- (c1) -- (ba)--(c2)--(a);
\draw [postaction={decorate}] (c1) -- (a);
\draw [postaction={decorate}] (c1) -- (ba);

\end{tikzpicture}	}
        \end{minipage}
    \end{center} 
    \caption{Illustrations for Case $2$--$4$ (from left to right).}
    \label{fig:claim3}
\end{figure}

{\bf Case $3$}: Suppose $\ori{c_1 a}$ and $\ori{b_a c_2}$ are directed edges, or $\ori{c_2 a}$ and $\ori{b_a c_1}$ are directed edges. Without loss of generality we assume the former. Note that by the minimality of $P$, $\ori{c_1 a}$ and $\ori{b_a c_2}$ can not both be in $\ori{P}$. We claim that $\ori{c_1 a}$ and $\ori{b_a c_2}$ also can not both be in $\ori{Q}$. Otherwise, by the minimality of $\ori{Q}$, either $\ori{c_1 b_a},\ori{ac_2}\in E(\ori{P})$ (if $\ind_Q(b_a)<\ind_Q(c_1)$) or $\ori{b_a c_1},\ori{c_2 a}\in E(\ori{P})$ (if $\ind_Q(b_a)>\ind_Q(c_1)$), which again gives a contradiction by the minimality of $P$.

Hence one of $\ori{c_1 a}$ and $\ori{b_a c_2}$ is in $E(\ori{P})$ and the other is in $E(\ori{Q})$. We claim that $\ori{c_1 a}\in E(\ori{P})$. Otherwise $\ori{b_a c_2}\in E(\ori{P})$, which implies that $c_2\in V(P)$. By the minimality of $P$ and since $a\in V(P)$, we then obtain that $\ori{c_2 a}\in E(\ori{P})$ and we are done by Case $2$. Thus $\ori{c_1 a}\in E(\ori{P})$ and $\ori{b_a c_2}\in E(\ori{Q})$. Again since we are not in Case 2, $\ori{c_2 a}$ is not a directed edge. We can then orient $a$ towards $c_2$ (if not already oriented). Since $c_2\in V(Q)$, we can then find some $b_1 \in B_{i+1}'\cup \{s_i\}$ such that $d_{\hh_{i+1}}(a,b_1)\leq 5$ by traveling along $\ori{Q}$. Now if $b_a c_1$ is undirected or oriented towards $c_1$, we can orient $b_a$ towards $c_1$ and let $b_2$ in (ii) be $b_a$. Otherwise, $\ori{c_1 b_a} \in E(\ori{Q})$. It follows that $c_1\in V(P)\cap V(Q)$. We can then find $b_2\in B_{i+1}'$ such that $d_{\hh_{i+1}}(b_2,a)\leq 5$ again by traveling along $\ori{Q}$ backwards. 

{\bf Case $4$}: Suppose $\ori{c_1 a}$ and $\ori{c_1 b_a}$ are directed edges, or $\ori{c_2 a}$ and $\ori{c_2 b_a}$ are directed edges. Without loss of generality, we assume the former. Similar to before, we have $c_1\in V(P)\cap V(Q)$ and (ii) holds by traveling along $\ori{Q}$ backwards. Moreover, since we are not in Case 2 and Case 3, we can orient $a$ towards $c_2$, and orient $c_2$ towards $b_a$ (if not already oriented) so that $d_{\hh_{i+1}}(a, b_a)\leq 2$. Hence (i) holds. 

{\bf Case $5$}: Suppose $\ori{ac_1}$ and $\ori{a c_2}$ are both oriented away from $a$. Then $a\in V(P)\cap V(Q)$ and we are done as before by traveling along $\ori{Q}$ (forwards or backwards).

{\bf Case $6$}: Suppose $\ori{c_1 b_a}$ and $\ori{c_2 b_a}$ are both oriented towards $b_a$. Then we have that one of the two edges is in $\ori{P}$ and the other is in $\ori{Q}$. Without loss of generality, assume $\ori{c_1 b_a}\in E(\ori{P})$ and $\ori{c_2 b_a}\in E(\ori{Q})$. Since $a\in V(P)$, it follows that $\ori{a c_1}\in E(\ori{P})$. Thus (i) holds by letting $b_1$ be $b_a$. Since we are not in Case 5, we can orient $c_2$ towards $a$ (if not already oriented). Since $c_2\in V(Q)$, (ii) holds by traveling along $\ori{Q}$ backwards.

{\bf Case $7$}: Suppose $\ori{a c_1}$ and $\ori{c_2 b_a}$ are directed edges, or $\ori{a c_2}$ and $\ori{c_1 b_a}$ are directed edges. The arguments are similar to Case 3, and we leave the details to the readers.

{\bf Case $8$}: Suppose $\ori{a c_1}$ and $\ori{b_a c_1}$ are directed edges, or $\ori{a c_2}$ and $\ori{b_a c_2}$ are directed edges. The arguments are similar to Case 4, and we leave the details to the readers.

Note that in all the above cases, it is easy to check that every vertex can be reached from every other vertex in $\hh_{i+1}$ by using edges in $E(\ori{P})\cup E(\ori{Q})\cup R_{i+1}$. Hence we can conclude that $\hh_{i+1}$ is strongly connected.
\end{proof}

By Claim \ref{cl:good-orientaiton}, there exists a strongly connected orientation of $\hh_{i+1}$ satisfying the conditions in Claim \ref{cl:good-orientaiton}. Since, by induction, $H_i$ is strongly connected, it follows that $H_{i+1}$ is also strongly connected.

In the next claim, we will show that $\diam(H_{i+1}) \leq (3+\epsilon)|S_{i+1}|$. Recall the notation that $H_{i+1} = H_{i} \cup H_{i+1}'$, $S_{i+1} = S_i\cup S_{i+1}'$, where $S_{i+1}'$ is the larger set of $A_{i+1}$ and $B_{i+1}$.
Note that by induction, we can assume that
    $\textrm{diam}(H_{i}) \leq (3+\epsilon)|S_{i}|$.

\begin{claim}\label{cl:new-to-new}
For any two distinct vertices $x_0,y_0 \in V(H_{i+1})$, 
\begin{equation}\label{eq:travel}
d_{H_{i+1}}(x_0,y_0)\leq 3|S_{i+1}| + 32 \leq (3+\epsilon)|S_{i+1}|.
\end{equation}
\end{claim}
\begin{proof}
We first justify the second inequality. Note that $|S_{i+1}|\geq |S_{i+1}'|\geq |A_{i+1}'|\geq \frac{1}{3}L(\epsilon) \geq \frac{1}{3} \cdot \frac{100}{\epsilon}$. Thus we have that
$$(3+\epsilon)|S_{i+1}|\geq 3|S_{i+1}|+ 32.$$
Hence we will proceed to justify the first inequality.
We remark that, to ensure the clarity of the proof, we make no effort to optimize the constant $32$, as it is not significant when $n$ is sufficiently large.
By induction, we know that $\textrm{diam}(H_{i}) \leq (3+\epsilon)|S_{i}|$. So we can assume that at least one of $x$ and $y$ are in $V(H_{i+1}')\backslash V(H_{i})$. Observe that to travel from $x_0$ to $y_0$, we may need to travel within $H_i$. 
Recall that $\hh_{i+1}$ is the graph obtained from $H_{i+1}:= H_i\cup H_{i+1}'$ by contracting $H_i$ into a single vertex $s_i$.
Observe that if without loss of generality $x_0\in V(H_i)$ and $y_0\in V(H_{i+1}')$, then 
$$d_{H_{i+1}}(x_0,y_0)\leq d_{H_i}(x_0, u_0) + d_{\hh_{i+1}}(s_i,y_0),$$
and 
$$d_{H_{i+1}}(y_0,x_0)\leq d_{\hh_{i+1}}(y_0, s_i)+ d_{H_i}(w_{q+1},x_0).$$
On the other hand, if both $x_0$ and $y_0$ are in $V(H_{i+1}')$, then 
$$d_{H_{i+1}}(x_0,y_0)\leq d_{\hh_{i+1}}(x_0,y_0) + (3+\epsilon)|S_i|.$$
Since $S_i \cap S_{i+1}' = \emptyset$, and  $\textrm{diam}(H_{i}) \leq (3+\epsilon)|S_{i}|$ (by the induction hypothesis), it is not hard to see that to prove Claim \ref{cl:new-to-new}, it suffices to show that for any two distinct vertices $x_0,y_0 \in V(\hh_{i+1})$,
\begin{equation}\label{eq:travel1}
    d_{\hh_{i+1}}(x_0,y_0)\leq 3|S_{i+1}'|+32 =  3\max\{|A_{i+1}|,|B_{i+1}|\}+32,
\end{equation}
which will then imply that $d_{\hh_{i+1}}(x_0,y_0)\leq (3+\epsilon)|S_{i+1}'|$.

To further simplify the case analysis, we will show that for any two distinct vertices $x_1,y_1 \in A_{i+1}\cup B_{i+1}\cup \{s_i\}$, 
\begin{equation}\label{eq:travel2}
d_{\hh_{i+1}}(x_1,y_1) \leq 3\max\{|A_{i+1}|,|B_{i+1}|\}+22.    
\end{equation}
Equation \eqref{eq:travel1} will follow from Equation \eqref{eq:travel2}, since for any distinct vertices $x_0, y_0 \in V(\hh_{i+1})$, there exist some vertices $x_1,y_1 \in A_{i+1} \cup B_{i+1}\cup \{s_i\}$ such that $d_{\hh_{i+1}}(x_0,x_1) \leq 5$ and $d_{\hh_{i+1}}(y_1,y_0)\leq 5$. Then we have that 
$$d_{\hh_{i+1}}(x_0,y_0) \leq d_{\hh_{i+1}}(x_0,x_1) + d_{\hh_{i+1}}(x_1,y_1) + d_{\hh_{i+1}}(y_1,y_0) \leq 3\max\{|A_{i+1}|,|B_{i+1}|\}+32.$$

In the remaining of the proof, we will show Equation \eqref{eq:travel2}. We will repeatedly use the fact that $|E(P)| = 3|A_{i+1}'|$ and $|E(Q)|\leq 3|B_{i+1}'|+2$ (see Figure \ref{fig:main_paths} for an illustration).

Let $x_1,y_1 \in A_{i+1}\cup B_{i+1}\cup \{s_i\}$ be arbitrary. Note that if $A_{i+1}'\backslash B_{i+1}'' =\emptyset$ or $B_{i+1}'\backslash A_{i+1}'' = \emptyset$, then $A_{i+1}'=B_{i+1}''$ or $B_{i+1}'= A_{i+1}''$. In the former case, we have that $B_{i+1} = B_{i+1}'\cup B_{i+1}'' = A_{i+1}'\cup B_{i+1}'$; in the later case, $A_{i+1} = A_{i+1}'\cup A_{i+1}'' = A_{i+1}'\cup B_{i+1}'$. In either case, we have that 
\begin{align*}
    d_{\hh_{i+1}}(x_1,y_1)  \leq |E(\ori{P})|+ |E(\ori{Q})|-1 & \leq 3|A_{i+1}'|+(3|B_{i+1}'|+2)-1\\                       
                       & \leq 3\max\{|A_{i+1}|,|B_{i+1}|\}+1.
\end{align*}
Therefore, we can assume that $A_{i+1}'\backslash B_{i+1}''\neq \emptyset$ and $B_{i+1}'\backslash A_{i+1}'' \neq \emptyset$.
Observe that if $x_1, y_1\neq s_i$, we could find $x,y\in (A_{i+1}'\backslash B_{i+1}'')\cup (B_{i+1}'\backslash A_{i+1}'')\cup \{s_i\}$, by traveling along $\ori{P}$ or $\ori{Q}$ (forwards or backwards) until reaching the first vertex in $(A_{i+1}'\backslash B_{i+1}'')\cup (B_{i+1}'\backslash A_{i+1}'')\cup \{s_i\}$, such that 
$$d_{\hh_{i+1}}(x_1,x)\leq 
\begin{cases}
    3|B_{i+1}''|+5 & \textrm{ if $x_1\in A_{i+1}'$},\\
    3|A_{i+1}''|+5 & \textrm{ if $x_1\in B_{i+1}'$};
\end{cases}$$ 
and similarly,
$$d_{\hh_{i+1}}(y,y_1)\leq 
\begin{cases}
    3|B_{i+1}''|+5 & \textrm{ if $x_1\in A_{i+1}'$},\\
    3|A_{i+1}''|+5 & \textrm{ if $x_1\in B_{i+1}'$}.
\end{cases}$$ 
In another word, there exists a path $P_{x_1 x}$ from $x_1$ to $x$ such that $V(\accentset{\circ}{P_{x_1x}})\cap S_{i+1}' \subseteq B_{i+1}''$ (if $x_1\in A_{i+1}'$) or $V(\accentset{\circ}{P_{x_1 x}})\cap S_{i+1}' \subseteq A_{i+1}''$ (if $x_1\in B_{i+1}'$).
Similarly, there exists a path $P_{y y_1}$ from $y$ to $y_1$ such that $V(\accentset{\circ}{P_{y y_1}})\cap S_{i+1}' \subseteq B_{i+1}''$ or $V(\accentset{\circ}{P_{y y_1}})\cap S_{i+1}' \subseteq A_{i+1}''$.
Note that $A_{i+1}''\subseteq B_{i+1}'\subseteq V(Q)$ and $B_{i+1}''\subseteq A_{i+1}'\subseteq V(P)$. Moreover, observe that
\begin{equation}\label{eq:travel3}
    d_{\hh_{i+1}}(x_1,y_1) \leq |E(P_{x_1 x}) \cup E(P_{x y})\cup E(P_{y y_1})|
\end{equation}

We will show Equation \eqref{eq:travel2} by case analysis on the locations of $x$ and $y$. Recall that $x,y \in (A'_{i+1}\backslash B_{i+1}'')\cup (B'_{i+1}\backslash A_{i+1}'')\cup \{s_i\}$. 
\begin{description}

\item {\bf Case 1}: $x\in (A'_{i+1}\backslash B_{i+1}'')\cup \{s_i\}$, $y\in (B_{i+1}'\backslash A_{i+1}'')\cup \{s_i\}$. 
By Claim \ref{cl:good-orientaiton} and the definition of $A_{i+1}'$ and $B_{i+1}'$, let $x'\in B_{i+1}'\cup \{s_i\}$ such that $d_{\hh_{i+1}}(x,x') \leq 5$, and let $y' \in A_{i+1}'\cup \{s_i\}$ such that $d_{\hh_{i+1}}(y',y)\leq 5$. If $\ind_P(y')\geq \ind_P(x)$, then
$$d_{\hh_{i+1}}(x,y)\leq |E(x\ori{P}y')|+d_{\hh_{i+1}}(y',y) \leq |E({P})|+5\leq  3|A_{i+1}'|+5.$$
On the other hand, if $\ind_Q(y)\geq \ind_Q(x')$, then
$$d_{\hh_{i+1}}(x,y) \leq d_{\hh_{i+1}}(x,x')+ |E(x'\ori{Q}y)| \leq 5+|E({Q})|\leq 5 + (3|B_{i+1}'|+2).$$
Since $A_{i+1}''\subseteq B_{i+1}'\subseteq V(Q)$ and $B_{i+1}''\subseteq A_{i+1}'\subseteq V(P)$, it is not hard to see that by Equation \eqref{eq:travel3} and the inequalities above, 
\begin{align*}
       d_{\hh_{i+1}}(x_1,y_1) & \leq |E(P_{x_1x}) \cup E(P_{x y})\cup E(P_{y y_1})|\\
      &\leq 3\max\{|A_{i+1}'|+|A_{i+1}''|,|B_{i+1}'|+|B_{i+1}''|\}+ 2+5\cdot 3 \\
    &\leq
    3\max\{|A_{i+1}|,|B_{i+1}|\}+ 17.
\end{align*}

Otherwise, we have that $\ind_P(y')< \ind_P(x)$ and $\ind_Q(y)< \ind_Q(x')$. See Figure \ref{fig:xypaths} for an illustration. Observe that in this case, we have $y_1\in V(Q)$ and $x_1\in V(P)$.
Let $P_{x x'}$ be the directed walk from $x$ to $x'$ with $|E(P_{x x'})|\leq 5$ and $P_{y' y}$ be the directed walk from $y'$ to $y$ with $|E(P_{y'y})|\leq 5$. 
Note that if $\ind_Q(y_1)\geq \ind_Q(x')$, then 
\begin{align*}
    d_{\hh_{i+1}}(x_1,y_1)&\leq |E(P_{x_1 x})|+ |E(P_{x x'})|+ d_{\ori{Q}}(x',y_1)\\
    &\leq (3|B_{i+1}''|+5)+ 5+ (3|B_{i+1}'|+2)\\
    &\leq 3|B_{i+1}|+12,
\end{align*}
and we are done. Similarly, if $\ind_P(x_1) \leq \ind_P(y')$, then 
\begin{align*}
    d_{\hh_{i+1}}(x_1,y_1)&\leq d_{\ori{P}}(x_1,y')+ |E(P_{y' y})|+ |E(P_{y y_1})|\\
    &\leq 3|A_{i+1}'|+ 5+ (3|A_{i+1}''|+5)\\
    &\leq 3|A_{i+1}|+10.
\end{align*}
Moreover, by our definitions of $x$ and $y$, $\ind_Q(y_1)\geq \ind_Q(y)$ and $\ind_P(x_1)\leq \ind_P(x)$. Hence we have that 
$$\ind_Q(y)\leq \ind_Q(y_1)< \ind_Q(x') \textrm{, and }
\ind_P(y')<\ind_P(x_1)\leq \ind_P(x).$$

\begin{figure}[htb]
	\begin{center}
        	\resizebox{10cm}{!}{\input{highway2.tikz}}
    \end{center} 
    \caption{Paths from $x_1$ to $y_1$.}
    \label{fig:xypaths}
\end{figure}

Let $P_1$ be $x\ori{P}u_p \ori{Q}y$, $P_2$ be $xP_{xx'}x' \ori{Q} s_i \ori{P}y'{P_{y'y}}y$.
Observe that $d_{\hh_{i+1}}(x,y)\leq \min\{|E(P_1)|, |E(P_2)|\}$, which implies that 
\begin{equation}\label{eq:travel5}
    d_{\hh_{i+1}}(x_1,y_1)\leq \min\{|E(P_1)|, |E(P_2)|\}+ |E(P_{y y_1})|+|E(P_{x_1 x})|.
\end{equation}

Observe that 
\begin{align*}
 |E(P_1)|+|E(P_2)| & \leq |E(P)|+|E(Q)|- (|E(P_{y y_1})|+|E(P_{x_1 x})|) +(|E(P_{xx'})|+|E(P_{y'y})|)  \\
 &\leq |E(P)|+|E(Q)|- (|E(P_{y y_1})|+|E(P_{x_1 x})|) +10.
\end{align*}
Thus, we have that 
\begin{align*}
    d_{\hh_{i+1}}(x_1,y_1) & \leq \min\{|E(P_1)|, |E(P_2)|\}+ |E(P_{y y_1})|+|E(P_{x_1 x})|\\
    &\leq \frac{1}{2}(|E(P_1)|+|E(P_2)|)+|E(P_{y y_1})|+|E(P_{x_1 x})|\\
    &\leq \frac{1}{2}\lp |E(P)|+|E(Q)|+|E(P_{y y_1})|+|E(P_{x_1 x})| \rp +5\\
    &\leq \frac{1}{2}\lp 3|A_{i+1}'|+3|B_{i+1}'|+2+ 3|A_{i+1}''|+5 + 3|B_{i+1}''|+5\rp+5\\
    &\leq 3\max\{|A_{i+1}|,|B_{i+1}|\}+ 11.
\end{align*}

This completes the proof of Case 1.

\item {\bf Case 2}: $x\in (B_{i+1}'\backslash A_{i+1}'')\cup \{s_i\}, y\in (A_{i+1}'\backslash B_{i+1}'')\cup \{s_i\}$. Case 2 is similar to Case 1. 

\item {\bf Case 3}: $x \in A_{i+1}'\backslash B_{i+1}''$, $y \in A_{i+1}'\backslash B_{i+1}''$. 
Thus $x,y\in V(P)$. If $\ind_P(x) \leq \ind_P(y)$, then 
$d_{\hh_{i+1}}(x,y)\leq 3|A_{i+1}'|$. Otherwise, by Claim \ref{cl:good-orientaiton}, there exists $x'\in B_{i+1}'\cup \{s_i\}$ such that $d_{\hh_{i+1}}(x,x')\leq 5$ and
$y' \in B_{i+1}'\cup \{s_i\}$ such that $d_{\hh_{i+1}}(y',y)\leq 5$. If $\ind_Q(x')\leq\ind_Q(y')$, then 
$$d_{\hh_{i+1}}(x,y)\leq d_{\hh_{i+1}}(x,x')+(3|B_{i+1}'|+2)+ d_{\hh_{i+1}}(y',y)\leq 3|B_{i+1}'|+12.$$
In both cases above, we have that 
$$d_{\hh_{i+1}}(x,y) \leq 3\max\{|A_{i+1}'|,|B_{i+1}'|\}+12.$$
Again, since $A_{i+1}''\subseteq B_{i+1}'\subseteq V(Q)$ and $B_{i+1}''\subseteq A_{i+1}'\subseteq V(P)$, we obtain that 
\begin{align*}
        d_{\hh_{i+1}}(x_1,y_1) & \leq |E(P_{x_1x}) \cup E(P_{x y})\cup E(P_{y y_1})|\\
      &\leq 3\max\{|A_{i+1}'|+|A_{i+1}''|,|B_{i+1}'|+|B_{i+1}''|\}+ 12+5\cdot 2 \\
    &\leq
    3\max\{|A_{i+1}|,|B_{i+1}|\}+ 22.
\end{align*}
Otherwise, we have $\ind_Q(y')<\ind_Q(x')$ and $\ind_P(x)> \ind_P(y)$. It follows by a similar argument in Case 1 that
$$d_{\hh_{i+1}'}(x_1,y_1)\leq 3\max\{|A_{i+1}|,|B_{i+1}|\}+22.$$

\item {\bf Case 4}: $x \in B_{i+1}'\backslash A_{i+1}''$, $y \in B_{i+1}'\backslash A_{i+1}''$. Case 4 is similar to Case 3. 
\end{description}
In all cases, we have that  
$d_{\hh_{i+1}}(x_1,y_1)\leq 3\max\{|A_{i+1}|,|B_{i+1}|\}+ 22$ for any two distinct vertices $x_1,y_1 \in A_{i+1}\cup B_{i+1}\cup \{s_i\}$. This completes the proof of Claim \ref{cl:new-to-new}, as discussed before.
\end{proof}

Now we are ready to complete the proof of Lemma \ref{lem:main-lemma} by justifying (b) of Lemma \ref{lem:main-lemma}.

\begin{claim}\label{cl:total_order}
For each $k\in [m]$, $\abs*{\bigcup_{v\in S_k} N[v]} \geq (\delta+1-3)|S_k| \geq (\delta-2)|S_k|$.
\end{claim}
\begin{proof}
Let $k\in [m]$. Recall that $S_k = \cup_{j=0}^k S_k'$.
By our construction, for any vertex $x \in S_{i}'$ and $y\in S_j'$ such that $i\neq j$, $N[x]\cap N[y] = \emptyset$. Hence it suffices to show that $\abs*{\bigcup_{v\in S_i'} N[v]} \geq (\delta-2)|S_i'|$ for all $i\in [m]$. Recall that $S_i'$ is the larger set of $A_{i}$ and $B_{i}$. We will consider the case when $S_i' = B_{i}$; the case $S_i'=A_{i}$ follows similar arguments (and is easier).   

Note that for every vertex $x\in B_i$, $|N[x]|\geq \delta+1$ since the minimum degree of $G$ is $\delta$. However, for two distinct vertices $x,y\in B_i$, $N[x]$ and $N[y]$ may overlap. We claim that if $N[x] \cap N[y] \neq \emptyset$, then $x$ and $y$ are almost non-overlapping, which implies that there exists a unique path of length at most two between $x$ and $y$. Recall that $B_i = B_i'\cup B_{i}''$. If $x,y\in B_i''\subseteq A_i'\backslash B_i'$, then $d_G(x,y)\geq 3$, which implies that $N[x]\cap N[y] = \emptyset$. If without loss of generality $x\in B_i'$, $y\in B_i''$ and $N[x]\cap N[y] \neq \emptyset$, then by the definition of $B_i''$, $y$ is almost non-overlapping with $x$. Hence we can assume that $x,y\in B_i'$. In this case, $x$ and $y$ are almost non-overlapping by the minimality of $P$ and $Q$, and the arguments are identical to the second part of the proof of Claim \ref{cl:short-cut}. Thus, we can conclude that if $N[x] \cap N[y] \neq \emptyset$, then there exists a unique path of length at most two between $x$ and $y$.

Now, construct an auxiliary multi-graph $\Gamma=(V(\Gamma), E(\Gamma))$ such that $V(\Gamma) = B_i$ and $xy\in E(\Gamma)$ if and only if $N[x]\cap N[y] \neq \emptyset$. Furthermore, an edge $xy\in E(\Gamma)$ has multiplicity $2$ if $xy\in E(G)$; $xy$ has multiplicity $1$ otherwise. It is not hard to see that 
\[\abs*{\bigcup_{v\in B_i} N[v]} \geq (\delta+1)|B_i| - |E(\Gamma)|.\]
Hence to show Claim \ref{cl:total_order}, it suffices to show that 
$\Delta(\Gamma)\leq 6$, which will then imply that $|E(\Gamma)|\leq 3|B_i|$.

Let $x\in V(\Gamma) = B_i = B_i'\cup B_i''$, where (recall that) $B_i' = $ 
$\{w_j: j\in [0, q-2] \textrm{ and $j\equiv 0$ (mod $3$)}\}$ and
$B_i'' = \{w \in A'_{i}\backslash B'_{i}: \textrm{ $w$ and $b$ are almost non-overlapping for all } b\in B'_{i}\}$. Suppose for contradiction that $d_{\Gamma}(x)\geq 7$.

\begin{figure}[htb]
	\begin{center}
        \begin{minipage}{.3\textwidth}
        		\resizebox{6cm}{!}{\begin{tikzpicture}[scale=1, Wvertex/.style={circle, draw=black, fill=white, scale=2}, bvertex/.style={circle, draw=black, fill=black, scale=0.4},rvertex/.style={circle, draw=red, fill=red, scale=0.2}, decoration={markings, mark= at position 0.8 with {\arrow[scale=2]{latex}}}]

\node [bvertex, label={[font=\normalsize] above:$b_1$}] (b1) at (4,2) {};

\node [bvertex, label={[font=\normalsize] above:$b_2$}] (b2) at (2,2) {};

\node [bvertex, label={[font=\normalsize] above:$b_3$}] (b3) at (0,2) {};

\node [bvertex, label={[font=\normalsize] above:$b_4$}] (b4) at (-2,2) {};

\node [bvertex, label={[font=\normalsize] above:$b_5$}] (b5) at (-4,2) {};

\node [bvertex, label={[font=\normalsize] below:$x\in B_i''\subseteq A_i'$}] (x) at (0,0) {};

\node [bvertex] (c1) at (2,1) {};
\node [bvertex] (c2) at (1,1) {};
\node [bvertex] (c3) at (0,1) {};
\node [bvertex] (c4) at (-1,1) {};
\node [bvertex] (c5) at (-2,1) {};

\node [bvertex] (d1a) at (2.667,2) {};
\node [bvertex] (d1b) at (3.334,2) {};
\node [bvertex] (d2a) at (0.667,2) {};
\node [bvertex] (d2b) at (1.334,2) {};
\node [bvertex] (d3a) at (-2.667,2) {};
\node [bvertex] (d3b) at (-3.334,2) {};
\node [bvertex] (d4a) at (-0.667,2) {};
\node [bvertex] (d4b) at (-1.334,2) {};

\draw [postaction={decorate}] (b5)--(c5);
\draw [postaction={decorate}] (c5)--(x);

\draw [postaction={decorate}] (b3)--(c3);
\draw [postaction={decorate}] (c3)--(x);

\draw (b1) -- (b2) -- (b3)--(b4)--(b5);
\draw (x) -- (b1);
\draw (x) -- (b2);
\draw (x) -- (b3);
\draw (x) -- (b4);
\draw (x) -- (b5);
\end{tikzpicture}	}
        \end{minipage}
            \hspace{2cm}
        \begin{minipage}{.3\textwidth}
        		\resizebox{5cm}{!}{\begin{tikzpicture}[scale=1, Wvertex/.style={circle, draw=black, fill=white, scale=2}, bvertex/.style={circle, draw=black, fill=black, scale=0.4},rvertex/.style={circle, draw=red, fill=red, scale=0.2}, decoration={markings, mark= at position 0.8 with {\arrow[scale=2]{latex}}}]

\node [bvertex, label={[font=\normalsize] above:$b_1$}] (b1) at (3,2) {};

\node [bvertex, label={[font=\normalsize] above:$b_2$}] (b2) at (1,2) {};

\node [bvertex, label={[font=\normalsize] above:$b_3$}] (b3) at (-1,2) {};

\node [bvertex, label={[font=\normalsize] above:$b_4$}] (b4) at (-3,2) {};

\node [bvertex, label={[font=\normalsize] below:$x\in B_i''\subseteq A_i'$}] (x) at (0,0) {};

\node [bvertex] (c) at (0.5,1) {};

\node [bvertex] (d1a) at (1.667,2) {};
\node [bvertex] (d1b) at (2.334,2) {};
\node [bvertex] (d2a) at (0.334,2) {};
\node [bvertex] (d2b) at (-0.334,2) {};
\node [bvertex] (d3a) at (-1.667,2) {};
\node [bvertex] (d3b) at (-2.334,2) {};

\draw [postaction={decorate}] (b4)--(x);
\draw [postaction={decorate}] (x)--(b1);

\draw (b1) -- (b2) -- (b3)--(b4);
\draw (x) -- (b1);
\draw (x) -- (b2);
\draw (x) -- (b3);
\draw (x) -- (b4);
\end{tikzpicture}	}
        \end{minipage}
    \end{center} 
    \caption{Cases when $x'\in B_i''$.}
    \label{fig:claim5a}
\end{figure}

Suppose first that $x\in B_i''$, and thus $x \in A_i'$. Observe that $x$ is not adjacent to any vertex $y\in B_i''$ since $d_G(x,y)\geq 3$ by the definition of $A_i'$. Thus since $d_{\Gamma}(x)\geq 7$, $x$ has at least four vertices $b_1, b_2, b_3, b_4\in B_i'$ adjacent to it in $\Gamma$ (and additionally $b_5$ if $x$ has at least five neighbors in $\Gamma$) listed in increasing order by their indices in $Q$. We first assume that $x$ has at least five neighbors in $\Gamma$ (which are all in $B_i'$). 
For $j\in [5]$, let $P_j$ be the path of length at most $2$ between $x$ and $b_j$.
Note that since $x\in V(\ori{P})$, then by the minimality of $P$, for each $j\in [5]$, if any edge in $E(\ori{P})\cap E(P_j)$ is oriented towards $x$, then the edge incident to $x$ in $P_j$ must be oriented towards $x$; similarly, if any edge in $E(P)\cap E(P_j)$ is oriented away from $x$, then the edge incident to $x$ in $P_j$ must be oriented away from $x$. 
Observe now that for $j<k \in [5]$ with $k-j\geq 2$, $d_G(b_j, b_k)\leq 4 < d_{Q}(b_j,b_k)$.
Thus it must follow that for each $j<k$ with $k-j\geq 2$, $b_j P_j x P_k b_k$ must have an edge that is in $\ori{P}$ which is oriented in the opposite direction of $b_j P_j x P_k b_k$. Additionally and similarly, for every $j\in [4]$ such that $|E(P_j)|=|E(P_{j+1})|=1$, $d_G(b_j, b_{j+1})\leq 2 < d_Q(b_j, b_{j+1})$. Thus it must follow that $b_j P_j x P_{j+1} b_{j+1}$ must have an edge that is in $\ori{P}$ which is oriented in the opposite direction of $b_j P_j x P_{j+1} b_{j+1}$ (see Figure \ref{fig:claim5a} for an illustration).
As a result, since $x$ has at least $5$ neighbors in $\Gamma$, by Pigeonhole Principle, $x$ must have two edges incident to it that are in $\ori{P}$ and are either both oriented towards it or both oriented away from it, leading to a contradiction.
Now we can assume that $x$ has exactly four vertices $b_1, b_2, b_3, b_4\in B_i'$ adjacent to it in $\Gamma$. 

Since $d_{\Gamma}(x)\geq 7$, at most one path in $P_1, P_2, P_3, P_4$ has length two. Observe that we obtain the same contradiction as above unless $P_1$, $P_4$ are paths of length $1$, $\ori{b_4x}, \ori{x b_1}\in E(\ori{P})$ and none of the edges in $P_2$ and $P_3$ are in $\ori{P}$. Now, let $Q':=Q-b_2Qb_3 + b_2 P_2 x P_3 b_3$. By the minimality of $Q$, $|V(Q')|\geq |V(Q)|$. Hence $d_{Q}(b_2, b_3)=3$.
Recall that $Q$ is picked (among all shortest paths from $u_p$ to $H_{i-1}$ that are consistent with $\ori{P}$) such that $|V(Q)\cap A_{i}'|$ is maximum. Moreover, observe that none of the internal vertices $z$ in $b_2\ori{Q}b_3$ is in $A_{i}'$, as otherwise, there exists a directed path of length at most $2$ from $z$ to $b_4$ (if $\ind_P(z)<\ind_P(x)$, or from $b_1$ to $z$ (if $\ind_P(z)>\ind_P(x)$, contradicting the minimality of $Q$.
Hence it follows that $|V(Q')\cap A_i'|>|V(Q)\cap A_i'|$, contradicting our choice of $Q$.

\begin{figure}[htb]
	\begin{center}
        \begin{minipage}{.3\textwidth}
        		\resizebox{4cm}{!}{\begin{tikzpicture}[scale=1, Wvertex/.style={circle, draw=black, fill=white, scale=2}, bvertex/.style={circle, draw=black, fill=black, scale=0.4},rvertex/.style={circle, draw=red, fill=red, scale=0.2}, decoration={markings, mark= at position 0.8 with {\arrow[scale=2]{latex}}}]

\node [bvertex, label={[font=\normalsize] above:$a_1$}] (a1) at (-2,0) {};

\node [bvertex, label={[font=\normalsize] above:$a_2$}] (a2) at (2,0) {};

\node [bvertex, label={[font=\normalsize] above:$x$}] (x) at (0,2) {};

\node [bvertex, label={[font=\normalsize] above:$b_1$}] (b1) at (2,2) {};

\node [bvertex, label={[font=\normalsize] above:$b_2$}] (b2) at (-2,2) {};

\node [bvertex] (c) at (1,1) {};

\draw [postaction={decorate}] (a1)--(x);
\draw [postaction={decorate}] (x)--(b1);

\draw (a2) -- (c) -- (x);
\draw (b2) -- (x);

\end{tikzpicture}	}
        \end{minipage}
            \hspace{2cm}
        \begin{minipage}{.3\textwidth}
        		\resizebox{5cm}{!}{\begin{tikzpicture}[scale=1, Wvertex/.style={circle, draw=black, fill=white, scale=2}, bvertex/.style={circle, draw=black, fill=black, scale=0.4},rvertex/.style={circle, draw=red, fill=red, scale=0.2}, decoration={markings, mark= at position 0.8 with {\arrow[scale=2]{latex}}}]

\node [bvertex, label={[font=\normalsize] above:$b_1$}] (b1) at (6,0) {};

\node [bvertex, label={[font=\normalsize] above:$b_2$}] (b2) at (3,0) {};

\node [bvertex, label={[font=\normalsize] above:$x$}] (x) at (0,0) {};

\node [bvertex, label={[font=\normalsize] above:$c_1$}] (c1) at (3,1.5) {};

\node [bvertex, label={[font=\normalsize] above:$c_2$}] (c2) at (1.5,1.5) {};

\node [bvertex] (d1a) at (4,0) {};
\node [bvertex] (d1b) at (5,0) {};
\node [bvertex] (d2a) at (1,0) {};
\node [bvertex] (d2b) at (2,0) {};

\draw (x) -- (b1);
\draw (x) -- (c2) -- (b2);
\draw (x) -- (c1) -- (b1);

\draw [postaction={decorate}] (c1)--(b1);
\draw [postaction={decorate}] (c2)--(b2);

\end{tikzpicture}	}
        \end{minipage}
    \end{center} 
    \caption{Cases when $x'\in B_i'\backslash A_i'$.}
    \label{fig:claim5b}
\end{figure}

Thus we can assume that $x\in B_i'\backslash A_i'$. We first note that there do not exist three neighbors $a_1, a_2, a_3$ of $x$ in $\Gamma$ such that $a_1, a_2, a_3 \in B_i''$; otherwise $a_1, a_2, a_3\in A_i'$ (assume they are listed in increasing order of their indices in $\ori{P}$), and we obtain that $d_G(a_1, a_3)\leq d_G(a_1, x) + d_G(x, a_3) \leq 2+2 = 4$, contradicting that $P$ is a shortest path between $a_1$ and $a_3$. 
Similarly, there do not exist two neighbors $a_1, a_2$ of $x$ in $\Gamma$ such that both $a_1 x$ and $a_2 x$ are edges in $G$. 

Note that since $d_\Gamma(x)\geq 7$, $x$ has at least two neighbors in $\Gamma$ that belong to $B_i'$. 
Suppose first that $x$ has exactly two neighbors $b_1,b_2\in B_i'$ in $\Gamma$ and $\ind_Q(b_1)< \ind_Q(b_2)$. Note that in this case, $d_{\Gamma}(x)\leq 7$ and thus $d_{\Gamma}(x)= 7$. It then follows that $xb_1, xb_2\in E(G)$, and there exist $a_1$, $a_2\in B_i''$ such that $xa_1 \in E(G)$ and there is a unique path $x c_2 a_2$ of length $2$ between $x$ and $a_2$. Observe that $d_G(b_1, b_2) =2 < d_{Q}(b_1, b_2)$, thus it follows that at least one of the edges $\ori{xb_1}$ and $\ori{b_2 x}$ must be in $E(\ori{P})$. Assume first that $\ori{xb_1}\in E(\ori{P})$, then $\ori{a_1x}\in E(\ori{P})$ by the minimality of $\ori{P}$. Thus, we have that $\ori{b_2 x}\notin E(\ori{P})$, contradicting that $d_{G-E(P)}(x,b_2)\geq 3$. The case when $\ori{b_2x} \in E(\ori{P})$ follows similar arguments.

Hence $x$ has at least three neighbors $b_1, b_2, b_3\in B_i'$ in $\Gamma$ listed in the increasing order of their indices in $Q$. Then since $x\in B_i'$, by Pigeonhole principle, at least two vertices in $\{b_1, b_2, b_3\}$ either both have greater indices in $\ori{Q}$ than $x$, or both have smaller indices in $\ori{Q}$ than $x$. Without loss of generality, assume that $b_1, b_2$ both have smaller indices in $\ori{Q}$ than $x$. Let $P_1, P_2$ be the paths of length at most $2$ from $x$ to $b_1, b_2$ respectively. 
Then since $d_G(x,b_j)\leq 2 < d_{G-E(P)}(x,b_j)$ for $j\in [2]$, at least one edge in each of $xP_1 b_1$ and $xP_2 b_2$ must be in $E(\ori{P})$ and oriented away from $x$ (to $b_1, b_2$ respectively). Call these two edges $e_1 \in E(P_1), e_2\in E(P_2)$ respectively. Clearly, these two edges can not both be incident to $x$. Moreover, $P_1$ and $P_2$ must both be paths of length $2$ by the minimality of $P$. Let $P_j = x c_j b_j$ for $j\in [2]$.
By the minimality of $P$, it must happen that $e_1 = \ori{c_1 b_1}$ and $e_2 = \ori{c_2 b_2}$. But now again by the minimality of $P$ and the fact that $d_G(c_1, c_2)\leq 2$, we have that 
either $\ori{b_2 c_1}\in E(\ori{P})$ (if $\ind_P(c_2) < \ind_P(c_1)$), or $\ori{b_1 c_2} \in E(\ori{P})$ $\ind_P(c_2) > \ind_P(c_1)$), contradicting that $x$ is almost non-overlapping with both $b_1$ and $b_2$.
\end{proof}

By Claims \ref{cl:new-to-new}, we have that 
$\ordiam(H_i)\leq 3|S_i|+32 \leq (3+\epsilon)|S_i|$ for every $i\in [m]$. 
Hence (a) of Lemma \ref{lem:main-lemma} holds.
Claim \ref{cl:total_order} implies that (b) of Lemma \ref{lem:main-lemma} holds.
Moreover, by the terminating condition of the algorithm, we have that $d(v,H_m) < L(\epsilon)$ for all $v\in V(G)$. Hence (c) holds. 
This completes the proof of the lemma. \qedhere
\end{proof}

\section{Proof of Theorem \ref{thm:closed-bound}.}

    For any fixed $\epsilon>0$, Lemma \ref{lem:main-lemma} gives a strongly connected orientation of a subgraph $H$ of $G$ such that every vertex in $V(G)\backslash V(H)$ is `close' (i.e., within $L(\epsilon)$ in graph distance) to $V(H)$ in $G$. Naturally we would like to extend the orientation of $H$ to an orientation of $G$ without increasing the diameter by too much. Such extension lemmas were initially given by Fomin et al. \cite{FMPR2004}, and by Bau and Dankelmann \cite{Bau-Dankelmann2015} when every vertex is at most distance $1$ or $2$ away from $H$ respectively. In \cite{Cochran2023+}, the first author extended these lemmas to graphs in which every vertex in $G$ is within a fixed constant distance $L$ from $V(H)$ in $G$.

\begin{lemma}\cite{Cochran2023+}\label{lem:extension}
	Let $G$ be a bridgeless graph, $H$ be a strongly connected orientation of a bridgeless subgraph of $G$ with $diam(H)=d$. Let $L$ be an integer such that $L\geq 2$ and for all $v\in V(G)$, $d_G(v,H)\leq L$. Then there exists a strongly connected orientation $\ori{G}$ of $G$ of diameter at most $d+4\binom{L+1}{2}$ that extends the orientation of $H$, i.e., $\ori{G}[V(H)] = H$.
\end{lemma}

Now we are ready to prove Theorem \ref{thm:closed-bound}.

\begin{proof}[Proof of Theorem \ref{thm:closed-bound}]
Let $\epsilon > 0$. Let $G$ be a bridgeless graph of order $n$ and minimum degree $\delta \geq 2$. By Lemma \ref{lem:main-lemma}, there exists some constant $L = L(\epsilon)$, a strongly connected orientation $H$ of $G$ and a vertex subset $S \subseteq V(H)$ such that 
$\diam(H) \leq (3+\epsilon)|S|$, $\abs*{\bigcup_{v\in S} N[v]} \geq (\delta-2)|S|$ and $d_G(v, H) \leq L$ for all $v\in V(G)$. By Lemma \ref{lem:extension}, there exists a strongly connected orientation of $G$ of diameter at most $(3+\epsilon)|S|+ 4\binom{L+1}{2}$ that extends the orientation of $H$.
Note that $(\delta-2)|S| \leq \abs*{\bigcup_{v\in S} N[v]}\leq n$. It follows that
\begin{align*}
    \ordiam(G) & \leq (3+\epsilon)|S|+ 4\binom{L+1}{2}\\
               & \leq (3+\epsilon) \frac{n}{\delta-2} + 4\binom{L+1}{2}.    
\end{align*}
This completes the proof of Theorem \ref{thm:closed-bound}.
\end{proof}

\end{document}